\documentclass[10pt]{article}

\usepackage{mathrsfs}
\usepackage{amsmath}
\usepackage{paralist}

\usepackage{amssymb}
\usepackage{color}
\usepackage{mathptmx}
\usepackage{verbatim}
\usepackage{epsfig}
\usepackage{CJK}
\usepackage{bm}
\usepackage{amsfonts}
\usepackage{amsthm}
\input{epsf.sty}
\usepackage{epsfig}
\def\be{\begin{equation}}
\def\ee{\end{equation}}
\def\ba{\begin{array}}
\def\ea{\end{array}}

\newtheorem{thm}{Theorem}[section]
\newtheorem{cor}[thm]{Corollary}
\newtheorem{lem}[thm]{Lemma}

\numberwithin{equation}{section}

\newcommand{\mi}{\mathbf{i}}

\def\be{\begin{equation}}
\def\ee{\end{equation}}
\def\br{\begin{eqnarray}}
\def\er{\end{eqnarray}}

\title{Birkhoff Normal Form for the Derivative Nonlinear Schr\"{o}dinger Equation}
\author{$\mbox{Jianjun \ Liu}$ \\
$\mbox{School of Mathematics, Sichuan University, Chengdu 610065, P R China}$}
\date{}
\begin{document}
\maketitle
{\bf Abstract}: This paper is concerned with the derivative nonlinear Schr\"{o}dinger equation with periodic boundary conditions. We obtain complete Birkhoff normal form of order six. As an application, the long time stability for solutions of small amplitude is proved.
\section{Introduction and Main Results}

In this paper, we consider the derivative nonlinear Schr\"{o}dinger
equation with periodic boundary conditions
\begin{equation}\label{20200523-2}
{\mi}u_t+u_{xx}+{\mi}(|u|^2u)_x=0,\hspace{12pt}x\in\mathbb{T}.
\end{equation}
The equation appears in plasma physics \cite{MioOMT76,Mj76,Sulem99}, and has infinitely many conservation laws \cite{KaupN78}.
%
%
Under the standard inner product on $L^2(\mathbb{T})$, (\ref{20200523-2}) can
be written in the form
\begin{equation}\label{20200523-3}
\frac{\partial{u}}{\partial{t}}=-\frac{d}{dx}\frac{\partial{H}}{\partial\bar{u}}
\end{equation}
with the real analytic Hamiltonian
\begin{equation}\label{20200523-4}
H=-\mi\int_\mathbb{T}u_x\bar{u}dx+\frac{1}{2}\int_\mathbb{T}|u|^4dx.
\end{equation}
Introduce for any $s\geq0$ the phase space
\begin{equation*}
\mathcal{H}^{s}_0=\{u\in{L}^2(\mathbb{T}):\hspace{6pt}\hat{u}(0)=0,\hspace{6pt}\|u\|^2_{s}=\sum_{j\in\mathbb{Z}_*}|\hat{u}(j)|^2|j|^{2s}<\infty\}
\end{equation*}
of complex valued functions on $\mathbb{T}$, where $\mathbb{Z}_*=\mathbb{Z}\setminus\{0\}$ and
\begin{equation*}
\hat{u}(j)=\int^{2\pi}_{0}u(x)e_{-j}(x)dx,\hspace{12pt}e_j(x)=\frac{1}{\sqrt{2\pi}}e^{{\mi}jx}.
\end{equation*}
To write (\ref{20200523-3}) in infinitely many coordinates, we make the
ansatz
\begin{equation*}
u(t,x)=\sum_{j\in\mathbb{Z}_*}q_j(t)e_j(x).
\end{equation*}
The coordinates are taken from the
Hilbert space $\ell^{2}_s$ of all complex-valued sequences
$q=(q_j)_{j\in\mathbb{Z}_*}$ with
\begin{equation*}
\|q\|_{s}^2=\sum_{j\in\mathbb{Z}_*}|q_j|^2|j|^{2s}<\infty.
\end{equation*}
Now (\ref{20200523-3}) can be rewritten as
\begin{equation}\label{20200523-5}
\dot{q_j}=-{\mi}j\frac{\partial{H}}{\partial\bar{q}_j}
\end{equation}
with the Hamiltonian
\begin{equation}\label{20200523-6}
H=\Lambda+G,
\end{equation}
where
\begin{equation}\label{20200523-7}
\Lambda=\sum_{j\in\mathbb{Z}_*}j|q_j|^2,
\end{equation}
\begin{equation}\label{20200523-8}
G=\frac{1}{4\pi}\sum_{\substack{j,k,l,m\in\mathbb{Z}_*\\j-k+l-m=0}}q_j\bar{q}_kq_l\bar{q}_m.
\end{equation}
In this paper the symplectic structure is
\begin{equation*}
-\mi\sum_{j\in\mathbb{Z}_*}jdq_j{\wedge}d\bar{q}_j,
\end{equation*}
the vector field for a Hamiltonian $F$ is
\begin{equation*}
X_F=-\mi\sum_{j\in\mathbb{Z}_*}j\Big(\frac{\partial{F}}{\partial\bar{q}_j}\frac{\partial}{\partial{q_j}}-\frac{\partial{F}}{\partial{q_j}}\frac{\partial}{\partial\bar{q}_j}\Big),
\end{equation*}
and the Poisson bracket for two Hamiltonian $H$, $F$ is
\begin{equation*}
\{H,F\}=-\mi\sum_{j\in\mathbb{Z}_*}j\Big(\frac{\partial{H}}{\partial{q_j}}\frac{\partial{F}}{\partial\bar{q}_j}-\frac{\partial{H}}{\partial\bar{q}_j}\frac{\partial{F}}{\partial{q_j}}\Big).
\end{equation*}

In \cite{G-L19,L-Y14}, in order to construct Cantor families of time quasi-periodic solutions for (\ref{20200523-2}) with higher order perturbations, a partial Birkhoff normal form of order four was calculated. In the following theorem, we transform the Hamiltonian (\ref{20200523-6}) into complete Birkhoff normal form of order four.
\begin{thm}\label{20200530-6}
There exists a real analytic symplectic coordinate transformation
$\Psi$ defined in a neighborhood of the origin of
$\ell^{2}_s$, $s>1$, which transforms the above Hamiltonian $H$ into complete Birkhoff normal form of order four. That is,
\begin{equation}\label{20200530-7}
H\circ\Psi=\Lambda+B+R,
\end{equation}
where $\Lambda$ is in (\ref{20200523-7}),
\begin{equation}\label{20200523-10}
B=-\frac{1}{4\pi}\sum_{j\in\mathbb{Z}_*}|q_j|^4+\frac{1}{2\pi}\Big(\sum_{j\in\mathbb{Z}_*}|q_j|^2\Big)^2,
\end{equation}
and $R$ is at least order 6 with
\begin{equation}\label{20200530-8}
\|X_R\|_{s-1}=O(\|q\|^5_{s}).
\end{equation}
Moreover, $R$ is of the form
\begin{equation*}\label{20200824-1}
R=\sum_{r\geq3}\sum_{(j_1,\cdots,j_{2r})\in\mathcal{M}_{r}}c(j_1,\cdots,j_{2r})q_{j_1}\bar{q}_{j_2}{\cdots}q_{j_{2r-1}}\bar{q}_{j_{2r}}
\end{equation*}
with its coefficients $c(j_1,\cdots,j_{2r})$ satisfying
\begin{equation}\label{20200824-2}
|c(j_1,\cdots,j_{2r})|<C^r\Big(\frac{j_2^*j_3^*{\cdots}j_{2r}^*}{j_1^*}\Big)^{\frac{1}{2}},
\end{equation}
where
\begin{equation*}\label{20200824-3}
\mathcal{M}_{r}:=\{(j_1,\cdots,j_{2r})\in\mathbb{Z}_*^{2r}:j_1-j_2+\cdots+j_{2r-1}-j_{2r}=0\}
\end{equation*}
is the zero momentum index set,
\begin{equation*}\label{20200824-4}
j_1^*{\geq}j_2^*\geq{\cdots}{\geq}j_{2r}^*
\end{equation*}
denotes the decreasing rearrangement of $\{|j_1|,|j_2|,\cdots,|j_{2r}|\}$, and $C>0$ is a constant.
\end{thm}

Then we eliminate all the resulting 6-order non normal form terms, and thus get complete Birkhoff normal form of order six in the following theorem.
\begin{thm}\label{20200823-2}
There exists a real analytic symplectic coordinate transformation
$\Phi$ defined in a neighborhood of the origin of
$\ell^{2}_s$, $s>3$, which transforms the above Hamiltonian $H\circ\Psi$ into complete Birkhoff normal form of order six. That is,
\begin{equation}\label{20200823-4}
H\circ\Psi\circ\Phi=\Lambda+B+K+\tilde{R},
\end{equation}
where $\Lambda$ is in (\ref{20200523-7}), $B$ is in (\ref{20200523-10}),
\begin{equation}\label{20200823-5}
K=-\frac{1}{8\pi^2}\sum_{j,k\in\mathbb{Z}_*,j{\neq}k}\frac{2j-k}{(j-k)^2}|q_j|^4|q_k|^2,
\end{equation}
and $\tilde{R}$ is at least order 8 with
\begin{equation}\label{20200823-6}
\|X_{\tilde{R}}\|_{s-1}=O(\|q\|^7_{s}).
\end{equation}
Moreover, $\tilde{R}$ is of the form
\begin{equation*}\label{20200824-5}
\tilde{R}=\sum_{r\geq4}\sum_{(j_1,\cdots,j_{2r})\in\mathcal{M}_{r}}\tilde{c}(j_1,\cdots,j_{2r})q_{j_1}\bar{q}_{j_2}{\cdots}q_{j_{2r-1}}\bar{q}_{j_{2r}}
\end{equation*}
with its coefficients $\tilde{c}(j_1,\cdots,j_{2r})$ satisfying
\begin{equation}\label{20200824-6}
|\tilde{c}(j_1,\cdots,j_{2r})|<\tilde{C}^r\Big(\frac{j_2^*j_3^*{\cdots}j_{2r}^*}{j_1^*}\Big)^{\frac{5}{2}},
\end{equation}
where $\tilde{C}>0$ is a constant.
\end{thm}

Compared with Theorem \ref{20200530-6}, the index $s$ of phase space in Theorem \ref{20200823-2} becomes more restrictive, and the bound of coefficients becomes larger. Thus, although Theorem \ref{20200530-6} is an intermediate step of Theorem \ref{20200823-2}, we still write them separately. In this way, we could try to study complete Birkhoff normal form of higher order. However, it becomes more and more complicated. On the other hand, although there are infinitely many conservation laws for the equation (\ref{20200523-2}), we don't know whether there is a global Birkhoff normal form as KdV equation in \cite{KapP03} and nonlinear Schr\"{o}dinger equation in \cite{GreK14}.

An application of Birkhoff normal form is to extract parameters by amplitude-frequency modulation in KAM theory. In this aspect, \cite{KukP96} is the pioneer work by Kuksin and P\"{o}schel, where Birkhoff normal form of order four is introduced to study nonlinear Schr\"{o}dinger equation. As mentioned above,  in \cite{G-L19,L-Y14} a partial Birkhoff normal form of order four was used to construct quasi-periodic solutions. By contrast, the complete Birkhoff normal form of order four or six is stronger, in the sense that frequencies are more twisted about parameters. For instance, it may help to remove some restrictions of selection of tangential frequencies in \cite{G-L19}.

Another application is to study long time stability of solutions of small amplitude. Actually, by Theorem \ref{20200530-6} and Theorem \ref{20200823-2}, we have the following corollary.
\begin{cor}\label{20200823-3}
For any $s>2$ there exists $\epsilon_{s}>0$ and $C_s>0$ such that if the
initial datum $u_0$ belongs to $\mathcal{H}^{s}_0$ and fulfills
\begin{equation*}\label{20200829-1}
\epsilon:=\|u_0\|_{s}\leq\epsilon_{s},
\end{equation*}
then the solution $u(t)$ of (\ref{20200523-2}) satisfies
\begin{equation}\label{20200829-2}
\|u(t)\|_{s}\leq{C_s}\epsilon
\end{equation}
for $|t|\leq\epsilon^{-r_s}$, where $r_s=4$ for $2<s\leq4$ and $r_s=6$ for $s>4$.
\end{cor}

Birkhoff normal form for long time stability of solutions of Hamiltonian partial differential equations has been widely investigated by many authors. For the case with bounded nonlinear vector field, see \cite{Bam03,Bam07,BamDGS07,BamG06,Bou00,C-L-Y16,FaouG13,GreIP09} for example; for the case with unbounded nonlinear vector field, see \cite{BertiD18,Del15,Y-Z14} for example. More recently, for the purpose to study long time stability without external parameters, rational normal form is introduced for nonlinear Schr\"{o}dinger equation in \cite{BernierFG18} and for generalized KdV and Benjamin-Ono equations in \cite{BernierG20}, where a key ingredient is to extract parameters by Birkhoff normal form of order six. So I think Theorem \ref{20200823-2} may be useful to study long time stability for more general derivative nonlinear Schr\"{o}dinger equation
\begin{equation}
{\mi}u_t+u_{xx}+{\mi}\Big(f(|u|^2)u\Big)_x=0,\hspace{12pt}x\in\mathbb{T},
\end{equation}
where $f(z)$ is a real analytic function with $f(0)=0$ and $f'(0)\neq0$.

In Section 2, we prove Theorem \ref{20200530-6}, Theorem \ref{20200823-2} and Corollary \ref{20200823-3}. For the proof of Theorem \ref{20200530-6}, we first split the 4-order Hamiltonian $G$ into its normal form part $B$ and non normal form part $Q$; then eliminate $Q$ by a symplectic coordinate transformation $\Psi$, which is the time-1-map of the flow of a Hamiltonian vector field $X_F$. In order to establish the regularity of $X_{F}$, we estimate the lower bound of $|j^2-k^2+l^2-m^2|$ for non-resonant indices $(j,k,l,m)$, seeing Lemma \ref{20200530-1}.

For the proof of Theorem \ref{20200823-2}, we first calculate the resulting 6-order Hamiltonian $R_6:=K+\tilde{K}+\tilde{Q}$, where $K$ is normal form part, $\tilde{K}$ is resonant non normal form part, and $\tilde{Q}$ is non-resonant part. Here, a key ingredient is to verify $\tilde{K}=0$, seeing (\ref{20200825-10}). Then formally as before, we eliminate $\tilde{Q}$ by another symplectic coordinate transformation $\Phi$, which is the time-1-map of the flow of a Hamiltonian vector field $X_{\tilde{F}}$. However, it is not easy to establish the regularity of $X_{\tilde{F}}$. Precisely, $\tilde{Q}$ consists of monomials $q_{j_1}\bar{q}_{j_2}q_{j_3}\bar{q}_{j_4}q_{j_5}\bar{q}_{j_6}$ with $j_1^2-j_2^2+j_3^2-j_4^2+j_5^2-j_6^2\neq0$, which implies
\begin{equation}\label{20200918-1}
|j_1^2-j_2^2+j_3^2-j_4^2+j_5^2-j_6^2|\geq1.
\end{equation}
However, the lower bound estimate (\ref{20200918-1}) is usually not enough owing to the unboundedness of the vector field $X_{\tilde{Q}}$. Assume $|j_1|\geq|j_3|\geq|j_5|$, $|j_2|\geq|j_4|\geq|j_6|$, $|j_1|{\geq}|j_2|$ without loss of generality. If $j_1{\neq}j_2$, or $j_1=j_2$, $j_1^*\leq100{j_3^*}^2$, we estimate the lower bound in Lemma \ref{20200830-5}. Otherwise, $j_1=j_2$, $j_1^*>100{j_3^*}^2$, we explicitly write out this type of reducible non-resonant part $\tilde{Q}_0$ of $\tilde{Q}$, seeing (\ref{20200827-3}), and then we find its coefficients are roughly no more than $\frac{1}{j_1^*}$, seeing (\ref{20200830-16}) with $j_1^*=n$. This is a key observation. Consequently, the vector field $X_{\tilde{Q}_0}$ is bounded, and the lower bound estimate (\ref{20200918-1}) is enough.

For the proof of Corollary \ref{20200823-3}, the case $2<s\leq4$ follows from Theorem \ref{20200530-6} and the case $s>4$ follows from Theorem \ref{20200823-2}. We only give the estimate of $\frac{d}{dt}\|q\|_s^2$ in detail. To this end, we estimate the bound of $j_1|j_1|^{2s}-j_2|j_2|^{2s}+\cdots+j_{2r-1}|j_{2r-1}|^{2s}-j_{2r}|j_{2r}|^{2s}$
for any $r\geq3$ and $(j_1,\cdots,j_{2r})\in\mathcal{M}_{r}$, seeing Lemma \ref{20200830-18}.

Appendix contains a lemma, seeing Lemma \ref{20200826-3}, which is used to prove $\tilde{K}=0$ as mentioned above. Actually, only the second conclusion (\ref{20200826-8}) is used here. The first conclusion (\ref{20200826-7}), besides being used in the proof of (\ref{20200826-8}), even has its own meaning. In view of (\ref{20200826-1}), for $j,k,l,m\in\mathbb{Z}$ with $j-k+l-m=0$ and $j,l\neq{k}$, we have $\mu(j,k,l)=\frac{1}{(j-k)(l-k)}=\frac{-2}{j^2-k^2+l^2-m^2}$. Thus (\ref{20200826-7}) could be used to calculate complete Birkhoff normal form of order six for the usual cubic nonlinear Schr\"{o}dinger
equation with periodic boundary conditions
\begin{equation}
{\mi}u_t+u_{xx}\pm|u|^2u=0,\hspace{12pt}x\in\mathbb{T}.
\end{equation}

\section{Proof}
In this section, for convenience the notation ``$\in\mathbb{Z}_*$" is usually abbreviated as ``$\neq0$" or omitted. Define the non-resonant index set
\begin{equation*}
\Delta=\{(j,k,l,m)\in\mathbb{Z}_*^4:j-k+l-m=0,j\neq{k,m}\},
\end{equation*}
and denote the decreasing rearrangement of $\{|j|,|k|,|l|,|m|\}$ by $j_1^*{\geq}j_2^*\geq{j_3^*}{\geq}j_4^*$.

\begin{lem}\label{20200530-1}
For $(j,k,l,m)\in\Delta$, we have
\begin{equation}\label{20200530-2}
|j^2-k^2+l^2-m^2|\geq\frac{\sqrt{j_1^*}^3}{2\sqrt{j_2^*j_3^*j_4^*}}.
\end{equation}
\end{lem}

\begin{proof}
Without loss of generality, we assume $j_1^*=|m|$. Ahead of the proof, we give a simple inequality for two positive integers $a,b$:
\begin{equation*}
2ab\geq a+b,
\end{equation*}
which will be frequently used. From $j-k+l-m=0$ we get
\begin{equation}\label{20200530-3}
j^2-k^2+l^2-m^2=-2(m-j)(m-l)=-2(m-j)(j-k).
\end{equation}
Notice that $j,k,l,m\neq0$ and $j,l\neq{k,m}$. Hence, on one hand,
\begin{equation}\label{20200530-4}
|j^2-k^2+l^2-m^2|=\frac{2|(m-j)(m-l)jl|}{|jl|}\geq\frac{(|m-j|+|j|)(|m-l|+|l|)}{2|jl|}\geq\frac{m^2}{2|jl|};
\end{equation}
on the other hand,
\begin{equation}\label{20200530-5}
|j^2-k^2+l^2-m^2|=\frac{2|(m-j)(j-k)k|}{|k|}\geq\frac{|(m-j)j|}{|k|}\geq\frac{|m|}{2|k|}.
\end{equation}
We conclude form (\ref{20200530-4}) (\ref{20200530-5}) the inequality (\ref{20200530-2}).
\end{proof}

\begin{proof}[\textbf{Proof of Theorem \ref{20200530-6}}]
The normal form part of $G$ is
(\ref{20200523-8}) with $j=k$ or $j=m$, which is just $B$ in (\ref{20200523-10}); the non normal form part of $G$ is
\begin{equation}\label{20200523-11}
Q=\frac{1}{4\pi}\sum_{(j,k,l,m)\in\Delta}q_j\bar{q}_kq_l\bar{q}_m.
\end{equation}
Then the Hamiltonian (\ref{20200523-6}) is written as
\begin{equation}\label{20200523-12}
H=\Lambda+B+Q.
\end{equation}
Obviously the functions $B$, $Q$ are analytic in $\ell^{2}_s$ with real value, and their vector fields $X_B$, $X_Q$ are analytic maps from $\ell^{2}_s$ into $\ell^{2}_{s-1}$ with
\begin{equation}\label{20200523-13}
\|X_B\|_{s-1},\hspace{3pt}\|X_Q\|_{s-1}=O(\|q\|^3_{s}).
\end{equation}
Define $F=\sum_{j,k,l,m\neq0}F_{jklm}q_j\bar{q}_kq_l\bar{q}_m$ by
\begin{equation}\label{20200530-9}
F_{jklm}=
\begin{cases}
\frac{\mi}{4\pi}\frac{1}{j^2-k^2+l^2-m^2},\hspace{12pt}&\text{for}\hspace{6pt}(j,k,l,m)\in\Delta,\\
0,&\text{otherwise}.
\end{cases}
\end{equation}
Then we have
\begin{equation}\label{20200530-10}
\{\Lambda,F\}+Q=0.
\end{equation}
Let $\Psi=X^1_{F}$ be the time-1-map of the flow of the Hamiltonian
vector field $X_F$, and then
\begin{eqnarray}\label{20200530-11}
H\circ\Psi&=&H\circ{X}^t_{F}|_{t=1}\nonumber\\
&=&\Lambda+\{\Lambda,F\}+B+Q\nonumber\\
& &+\int^1_0(1-t)\{\{\Lambda,F\},F\}\circ{X}^t_{F}dt+\int^1_0\{B+Q,F\}\circ{X}^t_{F}dt\nonumber\\
&=&\Lambda+B+R,
\end{eqnarray}
where
\begin{equation}\label{20200530-12}
R=\int^1_0\{B+tQ,F\}\circ{X}^t_{F}dt.
\end{equation}
For $(j,k,l,m)\in\Delta$, by Lemma \ref{20200530-1} we have
\begin{equation}\label{20200830-1}
|F_{jklm}|\leq\frac{1}{2\pi}(j_1^*)^{-\frac{3}{2}}(j_2^*j_3^*j_4^*)^{\frac{1}{2}}.
\end{equation}
From the coefficient estimate (\ref{20200830-1}), we deduce the estimates of vector field $X_F$ and its derivative $DX_F$ as follows
\begin{equation}\label{20200830-2}
\|X_F\|_{s}=O(\|q\|^3_{s}),
\end{equation}
\begin{equation}\label{20200830-3}
\|DX_F\|_{\ell^{2}_{s}\rightarrow\ell^{2}_{s}},\hspace{3pt}\|DX_F\|_{\ell^{2}_{s-1}\rightarrow\ell^{2}_{s-1}}=O(\|q\|^2_{s}).
\end{equation}
These establish the regularity of the vector field $X_{F}$. In view of (\ref{20200530-12}), the estimate of vector field $X_R$ in (\ref{20200530-8}) follows from (\ref{20200523-13}) (\ref{20200830-2}) (\ref{20200830-3}). Moreover, in view of (\ref{20200523-8}), the coefficient $\frac{1}{4\pi}$ of $G$ and thus $B$, $Q$ is no more than
\begin{equation}\label{20200830-4}
\frac{\sqrt{3}}{4\pi}(j_1^*)^{-\frac{1}{2}}(j_2^*j_3^*j_4^*)^{\frac{1}{2}}.
\end{equation}
Then by induction, starting with the coefficient estimates of $B$, $Q$ in (\ref{20200830-4}) and $F$ in (\ref{20200830-1}), we get the coefficient estimate of $R$ in (\ref{20200824-2}). This finishes the proof of Theorem \ref{20200530-6}.
\end{proof}

Define the non-resonant index set
\begin{equation}\label{20200827-5}
\tilde{\Delta}=\left\{(j_1,j_2,j_3,j_4,j_5,j_6)\in\mathbb{Z}_*^6
\Big|\substack{j_1-j_2+j_3-j_4+j_5-j_6=0\\j_1^2-j_2^2+j_3^2-j_4^2+j_5^2-j_6^2\neq0}\right\},
\end{equation}
and denote the decreasing rearrangement of $\{|j_1|,|j_2|,|j_3|,|j_4|,|j_5|,|j_6|\}$ by $j_1^*{\geq}j_2^*\geq{j_3^*}{\geq}j_4^*{\geq}j_5^*{\geq}j_5^*$.

\begin{lem}\label{20200830-5}
For $(j_1,j_2,j_3,j_4,j_5,j_6)\in\tilde{\Delta}$ with $|j_1|\geq|j_3|\geq|j_5|$, $|j_2|\geq|j_4|\geq|j_6|$, $|j_1|\geq|j_2|$, except the case
\begin{equation}\label{20200830-6}
j_1=j_2,\hspace{6pt}j_1^*>100{j_3^*}^2,
\end{equation}
we have
\begin{equation}\label{20200830-7}
|j_1^2-j_2^2+j_3^2-j_4^2+j_5^2-j_6^2|\geq\frac{{j_1^*}^3}{100(j_2^*j_3^*j_4^*j_5^*j_6^*)^2}.
\end{equation}
\end{lem}

\begin{proof}
If $j_1=j_2$, $j_1^*\leq100{j_3^*}^2$, then the inequality (\ref{20200830-7}) is trivial since its right hand is no more than 1; if $j_1=-j_2$, then by $j_1-j_2+j_3-j_4+j_5-j_6=0$ we have $j_1^*\leq2j_3^*$, and thus the inequality (\ref{20200830-7}) is also trivial; thus we assume $|j_1|>|j_2|$ in the following. From $j_1-j_2+j_3-j_4+j_5-j_6=0$ we get $j_1^*\leq5j_2^*$, so it is sufficient to prove
\begin{equation}\label{20200830-8}
|j_1^2-j_2^2+j_3^2-j_4^2+j_5^2-j_6^2|\geq\frac{j_1^*}{4{j_3^*}^2}.
\end{equation}
Without loss of generality, we assume $j_1^*>4{j_3^*}^2$. If $j_2^*=|j_2|$, then
\begin{equation}\label{20200830-10}
|j_1^2-j_2^2+j_3^2-j_4^2+j_5^2-j_6^2|\geq{j_1^*}^2-{j_2^*}^2-2{j_3^*}^2>\frac{1}{2}{j_1^*}>\frac{j_1^*}{4{j_3^*}^2};
\end{equation}
otherwise, $j_2^*=|j_3|$, then
\begin{equation}\label{20200830-9}
|j_1^2-j_2^2+j_3^2-j_4^2+j_5^2-j_6^2|\geq{j_1^*}^2+{j_2^*}^2-3{j_3^*}^2>\frac{1}{2}{j_1^*}^2>\frac{j_1^*}{4{j_3^*}^2}.
\end{equation}
\end{proof}

\begin{proof}[\textbf{Proof of Theorem \ref{20200823-2}}]
In view of (\ref{20200530-12}), we have
\begin{equation}\label{20200825-1}
R=R_6+R_{\geq8},
\end{equation}
where the 6-order term
\begin{equation}\label{20200825-2}
R_6=\{B,F\}+\frac{1}{2}\{Q,F\},
\end{equation}
and the higher order terms
\begin{equation}\label{20200825-3}
R_{\geq8}=\int^1_0(1-t)\{\{B,F\},F\}\circ{X}^t_{F}dt+\frac{1}{2}\int^1_0(1-t^2)\{\{Q,F\},F\}\circ{X}^t_{F}dt.
\end{equation}
In view of (\ref{20200523-10}) (\ref{20200523-11}) (\ref{20200530-9}), calculating directly, we have
\begin{equation}\label{20200825-4}
\{B,F\}=-\frac{1}{4\pi^2}\sum_{(j,k,l,m)\in\Delta}\frac{m}{j^2-k^2+l^2-m^2}q_j\bar{q}_kq_l\bar{q}_m|q_m|^2+c.c.,
\end{equation}
where the notation ``$c.c.$" stands for complex conjugation, and
\begin{eqnarray}\label{20200825-5}
& &\frac{1}{2}\{Q,F\}\nonumber\\
&=&\frac{1}{2}\sum_{m\neq0}\frac{\partial{Q}}{\partial{q_m}}\big({-\mi}m\frac{\partial{F}}{\partial\bar{q}_m}\big)+c.c.\nonumber\\
&=&\frac{1}{2}\sum_{m\neq0}\Bigg(\frac{1}{2\pi}\sum_{\substack{m_1-m_2+m_3=m\\m_1,m_3\neq{m_2}}}\bar{q}_{m_1}q_{m_2}\bar{q}_{m_3}\Bigg)\Bigg(\frac{m}{2\pi}\sum_{\substack{j-k+l=m\\j,l\neq{k}}}\frac{1}{j^2-k^2+l^2-m^2}q_j\bar{q}_kq_l\Bigg)+c.c.\nonumber\\
&=&-\frac{1}{16\pi^2}\sum_{\substack{j-k+l-m_1+m_2-m_3=0\\j,l\neq{k}\\m_1,m_3\neq{m_2}}}\tau(j,k,l)q_j\bar{q}_kq_l\bar{q}_{m_1}q_{m_2}\bar{q}_{m_3}+c.c.,
\end{eqnarray}
where
\begin{equation}\label{20200825-6}
\tau(j,k,l)=\frac{-2(j-k+l)}{j^2-k^2+l^2-(j-k+l)^2}=\frac{j-k+l}{(j-k)(l-k)}.
\end{equation}

The normal form part $K$ of $R_6$ is
(\ref{20200825-5}) with $k=m_2$ and $\{j,l\}=\{m_1,m_3\}$, i.e.,
\begin{eqnarray}\label{20200825-7}
K&=&-\frac{1}{8\pi^2}\sum_{\substack{j,l\neq{k}}}\tau(j,k,l)(2-\delta_{jl})|q_j|^2|q_k|^2|q_l|^2\nonumber\\
&=&-\frac{1}{8\pi^2}\sum_{\substack{j\neq{k}}}\tau(j,k,j)|q_j|^4|q_k|^2-\frac{1}{4\pi^2}\sum_{\substack{j{\neq}k,k{\neq}l,l{\neq}j}}\tau(j,k,l)|q_j|^2|q_k|^2|q_l|^2\nonumber\\
&=&-\frac{1}{8\pi^2}\sum_{\substack{j\neq{k}}}\frac{2j-k}{(j-k)^2}|q_j|^4|q_k|^2,
\end{eqnarray}
where $\delta_{jl}=
\begin{cases}
1,&\hspace{0pt}j=l\\
0,&\hspace{0pt}j\neq{l}
\end{cases}$ and the last equality follows from the fact
\begin{equation}\label{20200825-8}
\tau(j,k,l)+\tau(k,l,j)+\tau(l,j,k)=0.
\end{equation}

The resonant non normal form part $\tilde{K}$ of $R_6$ is
(\ref{20200825-5}) with $j^2-k^2+l^2-m_1^2+m_2^2-m_3^2=0$ and $\{j,l,m_2\}\neq\{k,m_1,m_3\}$. These restrictions imply $\{j,l,m_2\}\cap\{k,m_1,m_3\}=\emptyset$ actually. Define the resonant non normal form index set
\begin{equation}\label{20200825-9}
\mathcal{R}=\left\{(j_1,j_2,j_3,j_4,j_5,j_6)\in\mathbb{Z}_*^6
\bigg|\substack{j_1-j_2+j_3-j_4+j_5-j_6=0\\j_1^2-j_2^2+j_3^2-j_4^2+j_5^2-j_6^2=0\\\{j_1,j_3,j_5\}\cap\{j_2,j_4,j_6\}=\emptyset}\right\}.
\end{equation}
Then we could rewrite $\tilde{K}$ as
\begin{eqnarray}\label{20200825-10}
\tilde{K}&=&-\frac{1}{16\pi^2}\sum_{(j_1,j_2,j_3,j_4,j_5,j_6)\in\mathcal{R}}\Big(\tau(j_1,j_2,j_3)+\tau(j_4,j_5,j_6)\Big)q_{j_1}\bar{q}_{j_2}q_{j_3}\bar{q}_{j_4}q_{j_5}\bar{q}_{j_6}\nonumber\\
&=&-\frac{1}{8\pi^2}\sum_{(j_1,j_2,j_3,j_4,j_5,j_6)\in\mathcal{R}}\tau(j_1,j_2,j_3)q_{j_1}\bar{q}_{j_2}q_{j_3}\bar{q}_{j_4}q_{j_5}\bar{q}_{j_6}\nonumber\\
&=&-\frac{1}{72\pi^2}\sum_{(j_1,j_2,j_3,j_4,j_5,j_6)\in\mathcal{R}}\Bigg(\sum_{\substack{\{\alpha<\gamma\}\subset\{1,3,5\}\\\beta\in\{2,4,6\}}}\tau(j_\alpha,j_\beta,j_\gamma)\Bigg)q_{j_1}\bar{q}_{j_2}q_{j_3}\bar{q}_{j_4}q_{j_5}\bar{q}_{j_6}\nonumber\\
&=&0,
\end{eqnarray}
where the last equality follows from (\ref{20200826-8}) in Lemma \ref{20200826-3}.

The non-resonant part $\tilde{Q}$ of $R_6$ consists of (\ref{20200825-4}) and (\ref{20200825-5}) with $j^2-k^2+l^2-m_1^2+m_2^2-m_3^2\neq0$. Then we could rewrite $\tilde{Q}$ of the form
\begin{equation}\label{20200827-6}
\tilde{Q}=\sum_{(j_1,j_2,j_3,j_4,j_5,j_6)\in\tilde{\Delta}}\tilde{Q}_{j_1j_2j_3j_4j_5j_6}q_{j_1}\bar{q}_{j_2}q_{j_3}\bar{q}_{j_4}q_{j_5}\bar{q}_{j_6}.
\end{equation}
Define $\tilde{F}=\sum_{(j_1,j_2,j_3,j_4,j_5,j_6)\in\tilde{\Delta}}\tilde{F}_{j_1j_2j_3j_4j_5j_6}q_{j_1}\bar{q}_{j_2}q_{j_3}\bar{q}_{j_4}q_{j_5}\bar{q}_{j_6}$ by
\begin{equation}\label{20200827-7}
\tilde{F}_{j_1j_2j_3j_4j_5j_6}=\frac{\mi}{j_1^2-j_2^2+j_3^2-j_4^2+j_5^2-j_6^2}\tilde{Q}_{j_1j_2j_3j_4j_5j_6}.
\end{equation}
Then we have
\begin{equation}\label{20200827-8}
\{\Lambda,\tilde{F}\}+\tilde{Q}=0.
\end{equation}
Let $\Phi=X^1_{\tilde{F}}$ be the time-1-map of the flow of the Hamiltonian
vector field $X_{\tilde{F}}$, and then
\begin{eqnarray}\label{20200827-9}
H\circ\Psi\circ\Phi&=&(\Lambda+B+K+\tilde{Q}+R_{\geq8})\circ{X}^t_{\tilde{F}}|_{t=1}\nonumber\\
&=&\Lambda+\{\Lambda,\tilde{F}\}+B+K+\tilde{Q}\nonumber\\
&&+\int^1_0(1-t)\{\{\Lambda,\tilde{F}\},\tilde{F}\}\circ{X}^t_{\tilde{F}}dt+\int^1_0\{B+K+\tilde{Q},\tilde{F}\}\circ{X}^t_{\tilde{F}}dt+R_{\geq8}\circ{X}^1_{\tilde{F}}\nonumber\\
&=&\Lambda+B+K+\tilde{R},
\end{eqnarray}
where
\begin{equation}\label{20200827-10}
\tilde{R}=\int^1_0\{B+K+\tilde{Q},t\tilde{F}\}\circ{X}^t_{\tilde{F}}dt+R_{\geq8}\circ{X}^1_{\tilde{F}}.
\end{equation}

To complete the proof of Theorem \ref{20200823-2}, it is sufficient to prove the following coefficient estimate
\begin{equation}\label{20200827-11}
|\tilde{F}_{j_1j_2j_3j_4j_5j_6}|<C_1(j_1^*)^{-\frac{7}{2}}(j_2^*j_3^*j_4^*j_5^*j_6^*)^\frac{5}{2}
\end{equation}
with some constant $C_1>0$. Actually, by (\ref{20200827-11}), we deduce for $s>3$ the estimates of vector field $X_{\tilde{F}}$ and its derivative $DX_{\tilde{F}}$ as follows
\begin{equation}\label{20200830-11}
\|X_{\tilde{F}}\|_{s}=O(\|q\|^5_{s}),
\end{equation}
\begin{equation}\label{20200830-12}
\|DX_{\tilde{F}}\|_{\ell^{2}_{s}\rightarrow\ell^{2}_{s}},\hspace{3pt}\|DX_{\tilde{F}}\|_{\ell^{2}_{s-1}\rightarrow\ell^{2}_{s-1}}=O(\|q\|^4_{s}).
\end{equation}
These establish the regularity of the vector field $X_{\tilde{F}}$. Then in view of the formula of $\tilde{R}$ in (\ref{20200827-10}), by (\ref{20200530-8}) (\ref{20200830-11}) (\ref{20200830-12}) we could deduce the estimate of its vector field in (\ref{20200823-6}), and by (\ref{20200824-2}) (\ref{20200827-11}) we could deduce the estimate of its coefficients in (\ref{20200824-6}).

To prove (\ref{20200827-11}), by (\ref{20200827-7}) it is equivalent to prove
\begin{equation}\label{20200830-13}
|\tilde{Q}_{j_1j_2j_3j_4j_5j_6}|<C_1(j_1^*)^{-\frac{7}{2}}(j_2^*j_3^*j_4^*j_5^*j_6^*)^\frac{5}{2}|j_1^2-j_2^2+j_3^2-j_4^2+j_5^2-j_6^2|.
\end{equation}
Without loss of generality, we assume $|j_1|\geq|j_3|\geq|j_5|$, $|j_2|\geq|j_4|\geq|j_6|$, $|j_1|{\geq}|j_2|$. If $j_1{\neq}j_2$, or $j_1=j_2$, $j_1^*\leq100{j_3^*}^2$, then by Lemma \ref{20200830-5} it is sufficient to prove
\begin{equation}\label{20200830-14}
|\tilde{Q}_{j_1j_2j_3j_4j_5j_6}|<\frac{C_1}{100}(j_1^*)^{-\frac{1}{2}}(j_2^*j_3^*j_4^*j_5^*j_6^*)^\frac{1}{2},
\end{equation}
which is ensured by (\ref{20200824-2}) with taking $C_1{\geq}100C^3$; otherwise, $j_1=j_2$, $j_1^*>100{j_3^*}^2$, it is sufficient to prove
\begin{equation}\label{20200830-15}
|\tilde{Q}_{j_1j_2j_3j_4j_5j_6}|<\frac{C_1}{j_1^*},
\end{equation}
which will be checked by directly calculating this type of coefficients in the following. This type of reducible part $\tilde{Q}_0$ of $\tilde{Q}$ comes from (\ref{20200825-5}) with $j^2-k^2+l^2-m_1^2+m_2^2-m_3^2\neq0$ and $\{j,l,m_2\}\cap\{k,m_1,m_3\}\neq\emptyset$. The last restriction implies $m_2=k$ or $\{j,l\}\cap\{m_1,m_3\}\neq\emptyset$. For $m_2=k$, noticing the non-resonant condition $j^2+l^2-m_1^2-m_3^2\neq0$, we get
\begin{equation}\label{20200827-1}
-\frac{1}{16\pi^2}\sum_{\substack{(j,m_1,l,m_3)\in\Delta\\j,l,m_1,m_3\neq{k}}}\tau(j,k,l)q_jq_l\bar{q}_{m_1}\bar{q}_{m_3}|q_k|^2+c.c.,
\end{equation}
and for $\{j,l\}\cap\{m_1,m_3\}\neq\emptyset$, we get
\begin{equation}\label{20200827-2}
-\frac{1}{16\pi^2}\sum_{\substack{(j,k,m_2,m_1)\in\Delta\\k,m_2\neq{l}}}\tau(j,k,l)(4-2\delta_{jl}-2\delta_{lm_1})q_j\bar{q}_k\bar{q}_{m_1}q_{m_2}|q_l|^2+c.c..
\end{equation}
From (\ref{20200827-1}) (\ref{20200827-2}), we get
\begin{eqnarray}\label{20200827-3}
\tilde{Q}_0&=&-\frac{1}{16\pi^2}\sum_{\substack{(j,k,l,m)\in\Delta\\|n|>100\max\{j^2,k^2,l^2,m^2\}}}\big(\tau(j,n,l)+\tau(k,n,m)\big)q_j\bar{q}_kq_l\bar{q}_m|q_n|^2\nonumber\\
&&-\frac{1}{4\pi^2}\sum_{\substack{(j,k,l,m)\in\Delta\\|n|>100\max\{j^2,k^2,l^2,m^2\}}}\big(\tau(j,k,n)+\tau(k,j,n)\big)q_j\bar{q}_kq_l\bar{q}_m|q_n|^2\nonumber\\
&=&\frac{1}{16\pi^2}\sum_{\substack{(j,k,l,m)\in\Delta\\|n|>100\max\{j^2,k^2,l^2,m^2\}}}\big(4\tau(j,n,k)-\tau(j,n,l)-\tau(k,n,m)\big)q_j\bar{q}_kq_l\bar{q}_m|q_n|^2,
\end{eqnarray}
where
\begin{equation}\label{20200830-16}
|\tau(j,n,k)|,\hspace{6pt}|\tau(j,n,l)|,\hspace{6pt}|\tau(k,n,m)|<\frac{2}{n}.
\end{equation}
Thus, (\ref{20200830-15}) is ensured by taking $C_1{\geq}\frac{3}{4\pi^2}$. This finishes the proof of Theorem \ref{20200823-2}.
\end{proof}

For $(j_1,\cdots,j_{2r})\in\mathcal{M}_{r}$, denote
\begin{equation}\label{20200830-17}
\Omega_s(j_1,\cdots,j_{2r}):=j_1|j_1|^{2s}-j_2|j_2|^{2s}+\cdots+j_{2r-1}|j_{2r-1}|^{2s}-j_{2r}|j_{2r}|^{2s}.
\end{equation}

\begin{lem}\label{20200830-18}
For $(j_1,\cdots,j_{2r})\in\mathcal{M}_{r}$ and $s\geq1$, we have
\begin{equation}\label{20200830-19}
|\Omega_s(j_1,\cdots,j_{2r})|\leq(2s+1)(2r)^{s+2}{j_1^*}^s{j_2^*}^sj_3^*.
\end{equation}
\end{lem}

\begin{proof}
Without loss of generality, we assume $|j_1|\geq\cdots\geq|j_{2r-1}|$, $|j_2|\geq\cdots\geq|j_{2r}|$, $|j_1|{\geq}|j_2|$. If $j_2^*=|j_2|$ with $j_1j_2<0$, or $j_2^*=|j_3|$ with $j_1j_3>0$, then by $j_1-j_2+\cdots+j_{2r-1}-j_{2r}=0$ we have $j_1^*<2rj_3^*$ and thus (\ref{20200830-19}); otherwise, $j_2^*=|j_2|$ with $j_1j_2>0$, or $j_2^*=|j_3|$ with $j_1j_3<0$, then it is sufficient to prove
\begin{equation}\label{20200830-20}
{j_1^*}^{2s+1}-{j_2^*}^{2s+1}\leq(2s+1)(2r)^{s+1}{j_1^*}^s{j_2^*}^sj_3^*,
\end{equation}
which is ensured by $j_1^*-j_2^*<2rj_3^*$ and $j_1^*<2rj_2^*$.
\end{proof}

\begin{proof}[\textbf{Proof of Corollary \ref{20200823-3}}]
For $2<s\leq4$, by Theorem \ref{20200530-6} we get
\begin{equation}\label{20200830-21}
\frac{d}{dt}\|q\|_s^2=\{\|q\|_s^2,R\}={\mi}\sum_{r\geq3}\sum_{(j_1,\cdots,j_{2r})\in\mathcal{M}_{r}}c(j_1,\cdots,j_{2r})\Omega_s(j_1,\cdots,j_{2r})q_{j_1}\bar{q}_{j_2}{\cdots}q_{j_{2r-1}}\bar{q}_{j_{2r}}.
\end{equation}
Thus, by (\ref{20200824-2}) and Lemma \ref{20200830-18}, $\frac{d}{dt}\|q\|_s^2$ is bounded by
\begin{eqnarray}\label{20200830-22}
&&\sum_{r\geq3}(2s+1)(2r)^{s+2}C^r\sum_{(j_1,\cdots,j_{2r})\in\mathcal{M}_{r}}{j_1^*}^s{j_2^*}^s(j_3^*{\cdots}j_{2r}^*)^{\frac{3}{2}}|q_{j_1}\bar{q}_{j_2}{\cdots}q_{j_{2r-1}}\bar{q}_{j_{2r}}|\nonumber\\
&&\leq{C_2}\sum_{r\geq3}r^{s+2}C^r\|q\|_s^{2r}\nonumber\\
&&\leq{C_3}\|q\|_s^6
\end{eqnarray}
provided $\|q\|_s^2<\frac{1}{2C}$, where $C_2>0$, $C_3>0$ are constants depending on $s$. Consequently, (\ref{20200829-2}) is deduced. In the same way, this corollary with $s>4$ follows from theorem \ref{20200823-2}.
\end{proof}

\section{Appendix}
For any $x,y,z\in\mathbb{R}$ with $x,z{\neq}y$, define
\begin{equation}\label{20200826-1}
\mu(x,y,z)=\frac{1}{(x-y)(z-y)},
\end{equation}
\begin{equation}\label{20200826-2}
\tau(x,y,z)=\frac{x-y+z}{(x-y)(z-y)}.
\end{equation}

\begin{lem}\label{20200826-3}
For any $x_1,x_2,x_3,y_1,y_2,y_3\in\mathbb{R}$ with
\begin{equation}\label{20200826-4}
\{x_1,x_2,x_3\}\cap\{y_1,y_2,y_3\}=\emptyset,
\end{equation}
\begin{equation}\label{20200826-5}
x_1+x_2+x_3=y_1+y_2+y_3,
\end{equation}
\begin{equation}\label{20200826-6}
x_1^2+x_2^2+x_3^2=y_1^2+y_2^2+y_3^2,
\end{equation}
we have the following two equalities
\begin{equation}\label{20200826-7}
\uppercase\expandafter{\romannumeral1}:=\sum_{\{\alpha<\gamma\}\subset\{1,2,3\},\beta\in\{1,2,3\}}\mu(x_\alpha,y_\beta,x_\gamma)=0,
\end{equation}
\begin{equation}\label{20200826-8}
\uppercase\expandafter{\romannumeral2}:=\sum_{\{\alpha<\gamma\}\subset\{1,2,3\},\beta\in\{1,2,3\}}\tau(x_\alpha,y_\beta,x_\gamma)=0.
\end{equation}
\end{lem}
\begin{proof}
Observe that for every $t\in\mathbb{R}$, under the translation $x_1+t,x_2+t,x_3+t,y_1+t,y_2+t,y_3+t$, the conditions (\ref{20200826-4})-(\ref{20200826-6}) still hold true, and $\mu(x_\alpha,y_\beta,x_\gamma)$, $\tau(x_\alpha,y_\beta,x_\gamma)$ in the conclusions (\ref{20200826-7})(\ref{20200826-8}) satisfy
\begin{equation}\label{20200826-9}
\mu(x_\alpha+t,y_\beta+t,x_\gamma+t)=\mu(x_\alpha,y_\beta,x_\gamma),
\end{equation}
\begin{equation}\label{20200826-10}
\tau(x_\alpha+t,y_\beta+t,x_\gamma+t)=\tau(x_\alpha,y_\beta,x_\gamma)+t\mu(x_\alpha,y_\beta,x_\gamma).
\end{equation}
Thus, without loss of generality, we assume
\begin{equation}\label{20200826-11}
x_1+x_2+x_3=y_1+y_2+y_3=0
\end{equation}
in the following. Denoting
\begin{equation}\label{20200826-12}
N:=x_1^2+x_2^2+x_3^2=y_1^2+y_2^2+y_3^2,
\end{equation}
then calculating directly, we have
\begin{equation}\label{20200826-13}
x_1x_2+x_2x_3+x_3x_1=y_1y_2+y_2y_3+y_3y_1=-\frac{N}{2},
\end{equation}
\begin{equation}\label{20200826-14}
x_1^2x_2^2+x_2^2x_3^2+x_3^2x_1^2=y_1^2y_2^2+y_2^2y_3^2+y_3^2y_1^2=\frac{N^2}{4},
\end{equation}
\begin{equation}\label{20200826-15}
x_1^4+x_2^4+x_3^4=y_1^4+y_2^4+y_3^4=\frac{N^2}{2}.
\end{equation}
Denoting
\begin{equation}\label{20200826-16}
X:=x_1x_2x_3, \hspace{9pt}Y:=y_1y_2y_3,
\end{equation}
then calculating directly, we have
\begin{equation}\label{20200826-17}
x_1^3+x_2^3+x_3^3=3X, \hspace{9pt}y_1^3+y_2^3+y_3^3=3Y,
\end{equation}
\begin{equation}\label{20200826-18}
x_1^3x_2^3+x_2^3x_3^3+x_3^3x_1^3=3X^2-\frac{N^3}{8}, \hspace{9pt}y_1^3y_2^3+y_2^3y_3^3+y_3^3y_1^3=3Y^2-\frac{N^3}{8}.
\end{equation}

For every $\beta\in\{1,2,3\}$, we have
\begin{equation}\label{20200826-19}
(x_1-y_\beta)(x_2-y_\beta)(x_3-y_\beta)=X+\frac{N}{2}y_\beta-y_\beta^3,
\end{equation}
and then
\begin{eqnarray}\label{20200826-20}
&&\sum_{\{\alpha<\gamma\}\subset\{1,2,3\}}\mu(x_\alpha,y_\beta,x_\gamma)\nonumber\\
&=&\frac{1}{(x_1-y_\beta)(x_2-y_\beta)}+\frac{1}{(x_1-y_\beta)(x_3-y_\beta)}+\frac{1}{(x_2-y_\beta)(x_3-y_\beta)}\nonumber\\
&=&\frac{-3y_\beta}{X+\frac{N}{2}y_\beta-y_\beta^3},
\end{eqnarray}
\begin{eqnarray}\label{20200826-21}
&&\sum_{\{\alpha<\gamma\}\subset\{1,2,3\}}\tau(x_\alpha,y_\beta,x_\gamma)\nonumber\\
&=&\frac{x_1+x_2-y_\beta}{(x_1-y_\beta)(x_2-y_\beta)}+\frac{x_1+x_3-y_\beta}{(x_1-y_\beta)(x_3-y_\beta)}+\frac{x_2+x_3-y_\beta}{(x_2-y_\beta)(x_3-y_\beta)}\nonumber\\
&=&\frac{3y_\beta^2-N}{X+\frac{N}{2}y_\beta-y_\beta^3}.
\end{eqnarray}
Summing over $\beta\in\{1,2,3\}$, we get
\begin{equation}\label{20200826-22}
\uppercase\expandafter{\romannumeral1}=\frac{-3y_1}{X+\frac{N}{2}y_1-y_1^3}+\frac{-3y_2}{X+\frac{N}{2}y_2-y_2^3}+\frac{-3y_3}{X+\frac{N}{2}y_3-y_3^3},
\end{equation}
\begin{equation}\label{20200826-23}
\uppercase\expandafter{\romannumeral2}=\frac{3y_1^2-N}{X+\frac{N}{2}y_1-y_1^3}+\frac{3y_2^2-N}{X+\frac{N}{2}y_2-y_2^3}+\frac{3y_3^2-N}{X+\frac{N}{2}y_3-y_3^3}.
\end{equation}
Multiplying (\ref{20200826-22}) by a common denominator, we get
\begin{eqnarray}\label{20200826-24}
\bigg(-\frac{1}{3}\prod_{\alpha,\beta\in\{1,2,3\}}(x_\alpha-y_\beta)\bigg)\uppercase\expandafter{\romannumeral1}&=&y_1(X+\frac{N}{2}y_2-y_2^3)(X+\frac{N}{2}y_3-y_3^3)\nonumber\\
&+&y_2(X+\frac{N}{2}y_3-y_3^3)(X+\frac{N}{2}y_1-y_1^3)\nonumber\\
&+&y_3(X+\frac{N}{2}y_1-y_1^3)(X+\frac{N}{2}y_2-y_2^3),
\end{eqnarray}
in which the $X^2$ terms:
\begin{equation}\label{20200826-25}
X^2(y_1+y_2+y_3)=0;
\end{equation}
the $X$ terms:
\begin{eqnarray}\label{20200826-26}
&&XN(y_1y_2+y_2y_3+y_3y_1)-X(y_1y_2^3+y_1^3y_2+y_1y_3^3+y_1^3y_3+y_2y_3^3+y_2^3y_3)\nonumber\\
&&=-\frac{1}{2}N^2X+\frac{1}{2}N^2X=0;
\end{eqnarray}
the other terms:
\begin{eqnarray}\label{20200826-27}
&&\frac{3}{4}N^2y_1y_2y_3-Ny_1y_2y_3(y_1^2+y_2^2+y_3^2)+y_1y_2y_3(y_1^2y_2^2+y_2^2y_3^2+y_3^2y_1^2)\nonumber\\
&&=\frac{3}{4}N^2Y-N^2Y+\frac{1}{4}N^2Y=0.
\end{eqnarray}
This completes the proof of (\ref{20200826-7}). Multiplying (\ref{20200826-23}) by a common denominator, we get
\begin{eqnarray}\label{20200826-28}
\bigg(-\frac{1}{3}\prod_{\alpha,\beta\in\{1,2,3\}}(x_\alpha-y_\beta)\bigg)\uppercase\expandafter{\romannumeral2}&=&(\frac{N}{3}-y_1^2)(X+\frac{N}{2}y_2-y_2^3)(X+\frac{N}{2}y_3-y_3^3)\nonumber\\
&+&(\frac{N}{3}-y_2^2)(X+\frac{N}{2}y_3-y_3^3)(X+\frac{N}{2}y_1-y_1^3)\nonumber\\
&+&(\frac{N}{3}-y_3^2)(X+\frac{N}{2}y_1-y_1^3)(X+\frac{N}{2}y_2-y_2^3),
\end{eqnarray}
in which the $X^2$ terms:
\begin{equation}\label{20200826-29}
X^2(N-y_1^2-y_2^2-y_3^2)=0;
\end{equation}
the $X$ terms:
\begin{eqnarray}\label{20200826-30}
&&\frac{1}{3}XN^2(y_1+y_2+y_3)-\frac{1}{6}XN(y_1^3+y_2^3+y_3^3)-X(y_1^2y_2^3+y_1^3y_2^2+y_1^2y_3^3+y_1^3y_3^2+y_2^2y_3^3+y_2^3y_3^2)\nonumber\\
&&=0-\frac{1}{2}NXY+\frac{1}{2}NXY=0;
\end{eqnarray}
the other terms:
\begin{eqnarray}\label{20200826-31}
&&\frac{1}{12}N^3(y_1y_2+y_2y_3+y_3y_1)\nonumber\\
&&-\frac{1}{12}N^2\Big(2(y_1y_2+y_2y_3+y_3y_1)(y_1^2+y_2^2+y_3^2)+y_1y_2y_3(y_1+y_2+y_3)\Big)\nonumber\\
&&+\frac{1}{6}N\Big(3(y_1y_2+y_2y_3+y_3y_1)(y_1^2y_2^2+y_2^2y_3^2+y_3^2y_1^2)-(y_1^3y_2^3+y_2^3y_3^3+y_3^3y_1^3)\Big)\nonumber\\
&&-y_1^2y_2^2y_3^2(y_1y_2+y_2y_3+y_3y_1)\nonumber\\
&&=-\frac{1}{24}N^4+\frac{1}{12}N^4-(\frac{1}{24}N^4+\frac{1}{2}NY^2)+\frac{1}{2}NY^2=0.
\end{eqnarray}
This completes the proof of (\ref{20200826-8}).
\end{proof}

\textbf{Acknowledgement.} {The author is supported by NNSFC 11671280, NNSFC 11822108, Fok Ying Tong Education Foundation 161002.}

\end{document}